\documentclass[
aps,%
12pt,%
final,%
notitlepage,%
oneside,%
onecolumn,%
nobibnotes,%
nofootinbib,
superscriptaddress,%
noshowpacs,%
centertags]%
{revtex4}

\usepackage{amssymb,amsmath,physics}
\usepackage{amsthm}
\usepackage{subfigure, graphicx}
\usepackage{epstopdf}
\usepackage{float,comment}

\newtheorem{theorem}{Theorem}[section]

\newtheorem{condition}{Condition}[section]

\begin{document}
		
    \title{Asymptotic solution for three-dimensional reaction-diffusion-advection equation with periodic boundary conditions}

    \author{\firstname{Aleksei}~\surname{Liubavin}}
        \email{liubavinaleksei@math.ecnu.edu.cn}
        \affiliation{MSU-BIT-SMBU Joint Research Center of Applied Mathematics, Shenzhen MSU-BIT University, Shenzhen 518172, People's Republic of China}
        \affiliation{School of Mathematical Sciences, Key Laboratory of MEA(Ministry of Education) \& Shanghai Key Laboratory of PMMP, East China Normal University, 200241, Shanghai, China}

    \author{\firstname{Mingkang}~\surname{Ni}}
	\email{xiaovikdo@163.com}
        \affiliation{School of Mathematical Sciences, Key Laboratory of MEA(Ministry of Education) \& Shanghai Key Laboratory of PMMP, East China Normal University, 200241, Shanghai, China}
		
    \author{\firstname{Ye}~\surname{Zhang}}
	\email{ye.zhang@smbu.edu.cn}
        \affiliation{School of Mathematics and Statistics, Beijing Institute of Technology, 100081, Beijing, China}
    
    \author{\firstname{Dmitrii}~\surname{Chaikovskii}}\thanks{Corresponding author} 
	\email{dmitriich@smbu.edu.cn}
        \affiliation{MSU-BIT-SMBU Joint Research Center of Applied Mathematics, Shenzhen MSU-BIT University, Shenzhen 518172, People's Republic of China}

    \begin{abstract}
		In this study, we investigate the dynamics of moving fronts in three-dimensional spaces, which form as a result of in-situ combustion during oil production. This phenomenon is also observed in other contexts, such as various autowave models and the propagation of acoustic waves. Our analysis involves a singularly perturbed reaction-diffusion-advection type initial-boundary value problem of a general form. We employ methods from asymptotic theory to develop an approximate smooth solution with an internal layer. Using local coordinates, we focus on the transition layer, where the solution undergoes rapid changes. Once the location of the transition layer is established, we can describe the solution across the full domain of the problem. Numerical examples are provided, demonstrating the high accuracy of the asymptotic method in predicting the behaviors of moving fronts. 
    \end{abstract}	
\maketitle
	

\section{Introduction}\label{Introduction}

In this paper, we present an in-depth investigation of the phenomenon of moving fronts, with a particular focus on the scenario involving three spatial coordinates. Building upon the foundation established by recent studies of singularly perturbed equations, such as reaction-diffusion equations \cite{VolNef2006, BozNef2010}, one-dimensional reaction-diffusion-advection equations \cite{AntLevNef2014}, and their systems \cite{NefYag2015}, as well as two-dimensional reaction-diffusion-advection equations \cite{VolNefAnt2015}, we strive to expand the current understanding of this subject. The relevance of this research extends to a wide range of natural processes where moving fronts are observed, including fire spread \cite{LibTyaPei2003} and acoustic waves \cite{Rud2017, Nef2019}. Moreover, utilizing asymptotic analysis allows to predict the initial function. It is required for the formation of an autowave solution, which in turn enhances the efficiency of the dynamically adapted mesh \cite{LukVolNef2017, LukShiVol2018}.

The technique of asymptotic expansions plays a crucial role in mathematical analysis by proving the tools for existence and uniqueness of solutions in the case of singularly perturbed partial differential equations (PDEs), and it also offers a method for closely approximating these exact solutions. Moreover, the utilization of the asymptotic expansion method greatly simplifies the resolution of inverse problems, as demonstrated in the studies \cite{ChaZha2022, ChaLiuZha2023, ChaiYe2023}. This method, grounded in perturbation theory, facilitates a streamlined analysis of nonlinear systems. The versatility of asymptotic expansion is further underscored by its potential use in real-time control of oil production. In situ combustion for oil production, also known as fire flooding or underground combustion, is a technique for enhanced oil recovery. This technique involves igniting a portion of the oil reservoir and introducing air or oxygen to sustain a controlled underground combustion process. The propagation of the combustion front can be mathematically represented using the reaction-diffusion-advection equation, as detailed in the work \cite{Volk2011}. The heat generated by this combustion reduces the viscosity of the heavy oil, making it easier to extract. However, it is a complex process requiring careful control and monitoring to avoid environmental issues and ensure the safety of the operation. The integration of these mathematical approaches into operational strategies can lead to more effective and precise management of resource extraction processes.

We focus on the analysis of the general case for singularly perturbed reaction-diffusion-advection type initial-boundary value problem:
\begin{equation}\label{sec1.InitProb}
	\begin{cases}
		\displaystyle \mu\Delta u - \pdv{u}{t} = A(x,y,z)\pdv{u}{x} + B(x,y,z)\pdv{u}{y} - u \pdv{u}{z} + F(u,x,y,z), \\ 
        \displaystyle  u(x,y,z,0,\mu) = u_{init}(x,y,z,\mu), \\
		\displaystyle  u(x,y,0,t,\mu) = u^{0}(x,y), \quad u(x,y,a,t,\mu) = u^{a}(x,y), \\ 
        u(x,y,z,t,\mu) = u(x+L,y,z,t,\mu), \\
        u(x,y,z,t,\mu) = u(x,y+M,z,t,\mu), \quad x, y, z \in \bar{D}, \quad t\in[0,T].
	\end{cases}
\end{equation}
Here, $0 < \mu \leqslant \mu_0$ represents a small parameter, $\bar{D} = \{(x,y,z) : x\in\mathbb{R}, y\in\mathbb{R}, z\in[0,a]\}$, and $F(u,x,y,z) \in I_u \times \bar{D} \times [0,T]$, where $I_u$ is the region containing all quantities $u$. For this particular case, we are looking for L- and M-periodic solutions $u$ in regard to parameters $x$ and $y$ respectively. Thus, it is natural to assume that functions $A$, $B$, $F$, $u^{0}$, $u^{a}$, and $u_{init}$ are L-periodic with respect to $x$ and M-periodic with respect to $y$, and sufficiently smooth in their corresponding domains.

Employing asymptotic analysis, our objective is to establish the conditions under which a smooth solution to the problem defined in \eqref{sec1.InitProb} exists, and subsequently, to construct its asymptotic solution for efficiently approximating this solution.

This paper contains the following parts: Section \ref{SecAssump} lays out the primary assumptions and introduces the local coordinates. Then, Section \ref{section3} presents the asymptotic representation of solutions, followed by establishing coefficients for the outer functions and the inner functions of the asymptotic representation. We are able to characterize the surface $h(x,y,t)$ at the midpoint of the transition layer in Section \ref{subsecFuncH}. Finally, Section \ref{SecMainResults} concludes our research by formulating the main theorem, and Section \ref{secNumExample} presents a numerical example to illustrate practical usage.

\section{Assumptions and local coordinates}
\label{SecAssump}
The construction of an asymptotic solution requires the introduction of several necessary conditions.
\begin{condition}\label{C1}
    Assume that $u_{init}(x,y,z,\mu) = U_{n-1}(x,y,z,0) + \mathcal{O}(\mu^{n})$, $u_{init}$ $(x,y,0,\mu)=u^{0}(x,y)$, and $u_{init}(x,y,a,\mu)=u^{a}(x,y)$, where $U_{n-1}$ is the asymptotic solution of order $n-1$, which will be determined during the application of the asymptotic method (see Theorem \ref{th1}).
\end{condition}

From Condition \ref{C1}, it follows that the front should already be formed at $t=0$.

\begin{condition}\label{C2}
    For the partial differential equation
    \begin{equation}\label{sec1.DegEq}
	A(x,y,z)\pdv{u}{x} + B(x,y,z)\pdv{u}{y} - u \pdv{u}{z} + F(u,x,y,z) = 0,
    \end{equation} 
    produces a unique solution $\varphi^{(-)}(x,y,z)$ for the boundary condition $u(x,y,0) = u^{0}(x,y)$, and a solution $\varphi^{(+)}(x,y,z)$ for the boundary condition $u(x,y,a) = u^{a}(x,y)$. These solutions $\varphi^{(\mp)}$ should be sufficiently smooth within domain $\bar{D}$, which implies that they belong to the class $C^2$, and be L-periodic with $x$, and M-periodic with $y$. Moreover, they should satisfy the arrangement 
    \[
	\varphi^{(-)} < 0 < \varphi^{(+)}, \quad \text{for all } (x,y,z) \in \bar{D}.
    \]
\end{condition}

Expression \eqref{sec1.DegEq} represents a degenerate equation for the stationary case and can be obtained from the initial problem \eqref{sec1.InitProb} by setting $\mu=0$ and considering that the solution $u$ of \eqref{sec1.InitProb} is time-independent. The existence of the solutions $\varphi^{(\mp)}$ could be shown by employing the Local Existence Theorem from the book \cite{Eva2010}, which is based on the Characteristic Method (see Appendix for details).

\begin{condition}\label{C3}
	Define the correlations
	\[
		G_A^{(\mp)}(x,y,z) := \dfrac{A(x,y,z)}{\varphi^{(\mp)}(x,y,z)}, \quad          G_B^{(\mp)}(x,y,z) := \dfrac{B(x,y,z)}{\varphi^{(\mp)}(x,y,z)}.
	\]
	Assume that they satisfy the Lipschitz condition for variables $x$, $y$, and $z$ in the domain $\bar{D}$, i.e., 
        \begin{align*}
            &\abs{G_A^{(\mp)}(x_1,y_1,z_1) - G_A^{(\mp)}(x_2,y_2,z_2)} \leqslant K_A \sqrt{(x_1 - x_2)^2 + (y_1 - y_2)^2 + (z_1 - z_2)^2}, \\
            &\abs{G_B^{(\mp)}(x_1,y_1,z_1) - G_B^{(\mp)}(x_2,y_2,z_2)} \leqslant K_B \sqrt{(x_1 - x_2)^2 + (y_1 - y_2)^2 + (z_1 - z_2)^2},
        \end{align*}
       	where $K_A$ and $K_B$ are non-negative constants, and $(x_1, y_1, z_1)$ and $(x_2, y_2, z_2)$ are any arbitrary points in $\bar{D}$.
\end{condition}

This condition is important during the construction of the regular coefficients and corresponds to the characteristic equation \eqref{sec3.CharEq}. 

The main focus of this paper is centered around examining instances where the solution to the problem described in equation \eqref{sec1.InitProb} manifests as a propagating front. It should be noted that the surface $h(x,y,t)$ divides the domain $\bar{D}$ into two parts at any given point in time $t$ from the interval $[0,T]$, namely $\bar{D}^{(-)} = \{(x,y,z) : x \in \mathbb{R}, y \in \mathbb{R}, z \in [0, h(x,y,t)]\}$ and $\bar{D}^{(+)} = \{(x,y,z) : x \in \mathbb{R}, y \in \mathbb{R}, z \in [h(x,y,t), a]\}$. Based on previously mentioned conditions, we anticipate that the solution approximates $\varphi^{(-)}$ for the area $\bar{D}^{(-)}$ across all time points $t$, and closely aligns with $\varphi^{(+)}$ for the area $\bar{D}^{(+)}$. This situation results in a swift transition between them near the surface $h$. Such behavior is commonly denominated as the \emph{inner layer}.

Furthermore, our methodology assumes that it is possible to determine solution $u$ on the surface $h$, leading to the following definition:
\begin{equation}\label{sec1.halfSum}
	u = \varphi^{*}(x,y,h) := \dfrac{1}{2} \left(\varphi^{(-)}(x,y,h) + \varphi^{(+)}(x,y,h)\right).
\end{equation}

As our analysis focuses on the neighborhood around the transition point, we simplify it by introducing the local coordinates $(l,m,r)$ defined by
\begin{equation}\label{sec1.LocalCoor}
	x = l - r\alpha_x(h), \quad y = m - r\alpha_y(h), \quad z = h(l,m,t) - r\alpha_z(h).
\end{equation} 

The quantities $\alpha_x(h)$, $\alpha_y(h)$, and $\alpha_z(h)$ represent the cosines of the angles between the normal vector to the surface $h$ at any given $t$ and the coordinate axes $x$, $y$, and $z$, respectively. They are defined as follows:
\begin{align}\label{sec1.Angles}
		&\alpha_x(h) = \dfrac{h_x}{\sqrt{1+h_x^2+h_y^2}}, \quad \alpha_y(h) = \dfrac{h_y}{\sqrt{1+h_x^2+h_y^2}}, \nonumber\\ 
		&\alpha_z(h) = \dfrac{1}{\sqrt{1+h_x^2+h_y^2}}.
\end{align}

The values $l$ and $m$ correspond to the $x$ and $y$ coordinates of the point where the normal intersects the surface. Assuming the behavior of the traveling wave, we set $r > 0$ in the region $\bar{D}^{(+)}$, $r < 0$ in $\bar{D}^{(-)}$, and $r = 0$ on the $h$. Note that, derivatives $h_x$ and $h_y$ in the expressions \eqref{sec1.Angles} are evaluated at $x = l$ and $y = m$, respectively.

The next step in our preparations involves redefining the differential operator in the new coordinates $(l,m,r)$:
\begin{equation}\label{sec1.DiffOperXYZ}
	L_{\bar{D}}[u] := \mu \Delta u - \pdv{u}{t} - A(x,y,z)\pdv{u}{x} - B(x,y,z)\pdv{u}{y} + u\pdv{u}{z}.
\end{equation}

According to the work \cite{ChaiYe2023} it is possible to redefine the Laplacian operator $\Delta$ as
\begin{equation}\label{sec1.LapOPRLM}
	\begin{aligned}
		\Delta = \sum_{i \in \{x, y, z\}} (\pdv{r}{i}\pdv{r}+\pdv{l}{i}\pdv{l}+\pdv{m}{i}\pdv{m})\qty(\pdv{i})
	\end{aligned}
\end{equation}

Moreover, the time derivative in terms of local coordinates is represented as
\begin{equation}\label{sec1.tDerLoc}
	\pdv{u}{t} = \pdv{u}{t} - \pdv{u}{x}\pdv{x}{t} - \pdv{u}{y}\pdv{y}{t} - \pdv{u}{z}\pdv{z}{t},
\end{equation}
where 
\begin{equation*}
\frac{\partial}{\partial p} = \sum_{v \in \{r, l, m\}} \frac{\partial v}{\partial p} \frac{\partial}{\partial v}, \quad p \in \{x, y, z\}.
\end{equation*}

Considering the high rate of change near the inner layer region, it is only natural to introduce the stretched variable $\xi = r/\mu$. Equations \eqref{sec1.LapOPRLM} and \eqref{sec1.tDerLoc} allow us to define a differential operator for our local coordinate system
\begin{equation}\label{sec1.DiffOperLMXi}
	\begin{aligned}
		L_{h}[u] &:= \dfrac{1}{\mu} \qty(\pdv[2]{u}{\xi} + \alpha_z(A(l,m,h)h_x + B(l,m,h)h_y + h_t + u)\pdv{u}{\xi})\\
		& \quad - \pdv{u}{t} + \alpha_z^3(2h_x h_y h_{xy} - (h_x^2 + 1)h_{yy} - (h_y^2 + 1)h_{xx})\pdv{u}{\xi} \\
		& \quad + \alpha_z^2(h_x(h_t + B(l,m,h)h_y + u)- (A(l,m,h)(h_y^2 + 1)))\pdv{u}{l}\\
		& \quad + \alpha_z^2(h_y(h_t + A(l,m,h)h_x + u)- (B(l,m,h)(h_x^2 + 1)))\pdv{u}{m}\\ 
		& \quad + \sum_{i=1} \mu^i L_i[u].
	\end{aligned}
\end{equation}

The next important step for the analysis of the inner region is the consideration of the auxiliary equation, which will be derived in Section \ref{ZeroOrderInnerFunctions}:
\begin{equation}\label{sec1.AuxEq}
	\pdv[2]{\tilde{u}}{\xi} + \alpha_z(h_t + A(x,y,h)h_x + B(x,y,h)h_y + \tilde{u})\pdv{\tilde{u}}{\xi} = 0
\end{equation}
with $x, y, t$, and $h$ as parameters. This is commonly replaced with the corresponding system
\begin{equation}\label{sec1.AuxSys}
	\pdv{\tilde{u}}{\xi} = \Phi; \quad \pdv{\Phi}{\xi} = -\alpha_z(h_t + A(h)h_x + B(h)h_y + \tilde{u})\Phi.
\end{equation}

Dividing the second equation of \eqref{sec1.AuxSys} by the first one gives
\begin{equation}\label{sec1.PhiAuxUEq}
	\pdv{\Phi}{\tilde{u}} = -\alpha_z(h_t + A(h)h_x + B(h)h_y + \tilde{u}).
\end{equation}

It is noted that the points $(\varphi^{(\mp)}(h), 0)$ are stationary points of \eqref{sec1.AuxSys} on the phase plane $(\tilde{u}, \Phi)$. By integrating equation \eqref{sec1.PhiAuxUEq}, we obtain explicit expressions for the phase trajectories $\Phi^{(-)}(\tilde{u}, h)$ from the point $(\varphi^{(-)}, 0)$ as $\xi$ approaches $-\infty$, and for $\Phi^{(+)}(\tilde{u}, h)$ from the point $(\varphi^{(+)}, 0)$ as $\xi$ approaches $+\infty$:
\begin{equation}\label{sec1.PhiExp}
	\begin{aligned}
		\Phi^{(\mp)} &= \alpha_z((h_t + A(x,y,h)h_x + B(x,y,h)h_y)(\varphi^{(\mp)}(x,y,h) - \tilde{u}))\\
		& \quad + \alpha_z(\dfrac{1}{2}((\varphi^{(\mp)}(x,y,h))^2 - \tilde{u}^2)).
	\end{aligned}
\end{equation}

Thus, if we can find such $h=h_0(x,y,t)$, which satisfies
\begin{equation}\label{sec1.PhiLEqualPhiR}
	\Phi^{(-)}(\tilde{u},h_0) - \Phi^{(+)}(\tilde{u},h_0)=0,
\end{equation}
then there exists a phase trajectory on $(\tilde{u},\Phi)$ when $h=h_0$, such that it connects steady points $(\varphi^{(-)},0)$ and $(\varphi^{(+)},0)$.

Combining equations \eqref{sec1.PhiExp} and \eqref{sec1.PhiLEqualPhiR} provides us with the final condition:
\begin{condition}\label{C4}
	Assume that, there exists a surface $h=h_0(x,y,t)$, which is the solution to the problem
	\begin{equation} \label{sec1.TransFuncEq}
	  \begin{cases}
			h_t + A(x,y,h)h_x + B(x,y,h)h_y = -\dfrac{1}{2}(\varphi^{(-)}(x,y,h) + \varphi^{(+)}(x,y,h)),\\[2ex]
			h(x,y,t) = h(x+L,y,t) = h(x,y+M,t), \quad h(x,y,0)=h_{init}(x,y),
		\end{cases}
	\end{equation}
	where $h_{init}(x,y)$ is a function describing the initial position of the surface.
\end{condition}

The verification of the existence of such a solution is usually performed numerically in practice, or alternatively, we can apply Theorem \ref{characteristictheorem} from the appendix.

\section{Asymptotic representations} \label{section3}
Now we proceed with the basic asymptotic scheme. The first step is to approximate the solution of the problem \eqref{sec1.InitProb}. Due to the discontinuous nature, this approximation has the following form:
\begin{equation}\label{sec3.serU}
	U = 
	\begin{cases}
		U^{(-)} := \bar{u}^{(-)} + Q^{(-)}, \quad \text{in } \bar{D}^{(-)}, \\
		U^{(+)} := \bar{u}^{(+)} + Q^{(+)}, \quad \text{in } \bar{D}^{(+)}, 
	\end{cases}
\end{equation}
where functions $U^{(\mp)}$ and their derivatives are continuously connected on the surface $h(x,y,t)$. Here, expressions
\begin{equation}\label{sec1.regSer}
	\bar{u}^{(\mp)} = \bar{u}_0^{(\mp)}(x,y,z) + \mu\bar{u}_1^{(\mp)}(x,y,z) + \ldots.
\end{equation}
represent the position of outer functions, while functions
\begin{equation}\label{sec1.trLayerSer}
	Q^{(\mp)} = Q_0^{(\mp)}(\xi,l,m,h,t) + \mu Q_1^{(\mp)}(\xi,l,m,h,t) + \ldots\;.
\end{equation}
define the transition layer. Note that, $\xi$ is the scaling parameter introduced in the previous chapter and express $h(x,y,t)$ as
\begin{equation}\label{sec1.hSer}
	h = h_0(x,y,t) + \mu h_1(x,y,t) + \ldots.
\end{equation}

Given that we are seeking a smooth solution for all possible values of $t$, the following conditions must hold:
\begin{equation}\label{sec1.UConnect}
	U^{(-)}(l,m,h,t,\mu) = U^{(+)}(l,m,h,t,\mu) = \varphi^{*}(l,m,h),
\end{equation}
\begin{equation}\label{sec1.DerUConnect}
	\pdv{U^{(-)}(l,m,h,t,\mu)}{n} = \pdv{U^{(+)}(l,m,h,t,\mu)}{n},
\end{equation}
where $\varphi^{*}(x,y,h)$ was defined in \eqref{sec1.halfSum} and $\displaystyle\pdv{n}$ represents the derivative in the direction normal to the surface $h$.

\subsection{Outer functions}
The next step in our research involves substituting the series \eqref{sec1.regSer} into the initial problem \eqref{sec1.InitProb}
\begin{equation}\label{sec3.RegInitEq}
	\begin{aligned} 
		\mu\qty(\pdv[2]{\bar{u}^{(\mp)}}{x} + \pdv[2]{\bar{u}^{(\mp)}}{y} + \pdv[2]{\bar{u}^{(\mp)}}{z}) &= A(x,y,z)\pdv{\bar{u}^{(\mp)}}{x} + B(x,y,z)\pdv{\bar{u}^{(\mp)}}{y}\\ 
		& \quad - \bar{u}^{(\mp)} \pdv{\bar{u}^{(\mp)}}{z} + F(\bar{u}^{(\mp)},x,y,z).
	\end{aligned}
\end{equation}
Note that $\bar{u}$ does not depend on the parameter $t$, which simplifies the initial PDE \eqref{sec1.InitProb} to a stationary equation.
We now expand functions on the right-hand side of equations \eqref{sec3.RegInitEq} into Taylor series by powers of $\mu$. Equating coefficients of the same powers of $\mu$ decomposes \eqref{sec3.RegInitEq} into a set of differential equations. These help us define all elements $\bar{u}_i^{(\mp)}$ from the series \eqref{sec1.regSer}. To solve these equations, we also need to construct the appropriate boundary conditions in a similar manner.

Accordingly, the coefficients of the term $\mu^0$ yield the following expressions:
\[
	A\pdv{\bar{u}_0^{(\mp)}}{x} + B\pdv{\bar{u}_0^{(\mp)}}{y} - \bar{u}_0^{(\mp)} \pdv{\bar{u}_0^{(\mp)}}{z} + F(\bar{u}_0^{(\mp)}) = 0.
\]
These equations are equivalent to \eqref{sec1.DegEq}. According to Condition \ref{C2}, they have $L$-periodic solutions by $x$ and $M$-periodic solutions by $y$, $\varphi^{(\mp)}(x,y,z)$, with corresponding conditions
\[
	\varphi^{(-)}(x,y,0) = u^{0}(x,y), \quad \varphi^{(+)}(x,y,a) = u^{a}(x,y).
\]
Similarly, the quantities $\bar{u}_i^{(\mp)},\, i=1,2\ldots$ can be obtained from the equation
\begin{multline}\label{sec3.Regu_iEq}
    \displaystyle A(x,y,z)\pdv{\bar{u}_i^{(\mp)}}{x} + B(x,y,z)\pdv{\bar{u}_i^{(\mp)}}{y} - \varphi^{(\mp)}\pdv{\bar{u}_i^{(\mp)}}{z}  \\
		+ W^{(\mp)}(x,y,z)\bar{u}_i^{(\mp)} = \bar{f}_i^{(\mp)}(x,y,z), 
\end{multline}
with boundary conditions $\bar{u}_i^{(\mp)}(x,y,0) = 0, \quad \bar{u}_i^{(\mp)}(x,y,a) = 0,$ and periodic conditions $\bar{u}_i^{(\mp)}(x+L,y,z) = \bar{u}_i^{(\mp)}(x,y,z)$, $\bar{u}_i^{(\mp)}(x,y,z) = \bar{u}_i^{(\mp)}(x,y+M,z)$. Function $W^{(\mp)}(x,y,z)$ contains the following derivatives: 
\[
	W^{(\mp)} = -\pdv{\varphi^{(\mp)}(x,y,z)}{z} + \pdv{F(\varphi^{(\mp)},x,y,z)}{u},
\]
and  $\bar{f}_i^{(\mp)}(x,y,z)$ are some known functions, particularly, 
\[
	\bar{f}_1^{(\mp)}(x,y,z) = \pdv[2]{\varphi^{(\mp)}}{x} + \pdv[2]{\varphi^{(\mp)}}{y} + \pdv[2]{\varphi^{(\mp)}}{z}.
\]

Boundary value problems \eqref{sec3.Regu_iEq} involve linear differential equations. Therefore, we can write down the characteristic equations
\begin{equation}\label{sec3.CharEq}
	\begin{cases}
		\displaystyle\dv{x}{z} = -\dfrac{A(x,y,z)}{\varphi^{(\mp)}(x,y,z)}, \quad \dv{y}{z} = -\dfrac{B(x,y,z)}{\varphi^{(\mp)}(x,y,z)}, \\
		\qty(\bar{f}_i^{(\mp)}(x,y,z) - W^{(\mp)}(x,y,z)\bar{u}_i^{(\mp)})\dd{y} = - \varphi^{(\mp)}(x,y,z) \dd{\bar{u}_i^{(\mp)}}.
	\end{cases}
\end{equation}

According to Condition \ref{C3}, within each system in \eqref{sec3.CharEq}, there exist first integrals
\begin{equation}\label{sec3.FirstInt}
	\Psi^{(\mp)} = C_1^{(\mp)}.
\end{equation}
Accordingly, on the segment $z \in [0, a]$, there exist functions $x = X^{(\mp)}(z, C_1^{(\mp)})$ and $y = Y^{(\mp)}(z, C_1^{(\mp)})$, which are solutions to the equations in the systems \eqref{sec3.CharEq}.

Thus, solving the equations
\[
	\dv{\bar{u}_i^{(\mp)}}{z} = -\dfrac{\bar{f}_i^{(\mp)}(X^{(\mp)}, Y^{(\mp)}, z) - W^{(\mp)}(X^{(\mp)}, Y^{(\mp)}, z)\bar{u}_i^{(\mp)}}{\varphi^{(\mp)}(X^{(\mp)}, Y^{(\mp)}, z)}
\]
with initial conditions $\bar{u}_i^{(\mp)}(x, y, 0) = 0$ and $\bar{u}_i^{(\mp)}(x, y, a) = 0$ allows us to obtain expressions for $\bar{u}_i^{(\mp)}(x, y, z)$ as follows
\begin{multline*}
	\bar{u}_i^{(\mp)}(X^{(\mp)},Y^{(\mp)},z) = -\int_{0,a}^{z} \dfrac{\bar{f}_i^{(\mp)}(X^{(\mp)}\qty(z_1,C_1^{(\mp)}),Y^{(\mp)}\qty(z_1,C_1^{(\mp)}),z_1)}{\varphi^{(\mp)}(X^{(\mp)}\qty(z_1,C_1^{(\mp)}),Y^{(\mp)}\qty(z_1,C_1^{(\mp)}),z_1)}\\
	\times \exp \qty(\int_{z_1}^{z} \dfrac{W^{(\mp)}(X^{(\mp)}\qty(z_2,C_1^{(\mp)}),Y^{(\mp)}\qty(z_2,C_1^{(\mp)}),z_2)}{\varphi^{(\mp)}(X^{(\mp)}\qty(z_2,C_1^{(\mp)}),Y^{(\mp)}\qty(z_2,C_1^{(\mp)}),z_2)}\dd{z_2}) \dd{z_1}.
\end{multline*}

\subsection{Inner layer functions}
Now we shift our attention to the inner layer area. Substituting \eqref{sec3.serU} into \eqref{sec1.InitProb} and splitting it by the powers of $\mu$ provides us with equations
\begin{equation}\label{sec4.TranLayerEq}
	\begin{aligned}
		\dfrac{1}{\mu} &\qty(\pdv[2]{Q^{(\mp)}}{\xi} + \alpha_z((h_t + A(h)h_x + B(h)h_y + \bar{u}^{(\mp)} + Q^{(\mp)}))\pdv{Q^{(\mp)}}{\xi})\\
		&- \pdv{Q^{(\mp)}}{t} + \alpha_z^3(2h_x h_y h_{xy} - (h_x^2 + 1)h_{yy} - (h_y^2 + 1)h_{xx})\pdv{Q^{(\mp)}}{\xi}\\
		&+ \alpha_z^2(h_x(h_t + B(h)h_y + \bar{u}^{(\mp)} + Q^{(\mp)}) - A(h)(h_y^2 + 1))\pdv{Q^{(\mp)}}{l}\\
		&+ \alpha_z^2(h_y(h_t + A(h)h_x + \bar{u}^{(\mp)} + Q^{(\mp)}) - B(h)(h_x^2 + 1))\pdv{Q^{(\mp)}}{m}\\
		&+ \sum_{i=1} \mu^i L_i[Q^{(\mp)}] = -Q^{(\mp)}\pdv{\bar{u}^{(\mp)}}{z} + \hat{F}^{(\mp)}(\xi,l,m,t,\mu). 
	\end{aligned} 
\end{equation}
Here we denote $A(h) := A(l,m,h)$, $B(h) := B(l,m,h)$, 
\begin{align*}
	\hat{F}^{(\mp)} &:= F(\bar{u}^{(\mp)}+Q^{(\mp)}, l - r\cos(\alpha_x), m - r\cos(\alpha_y)) \\ 
	& \quad -F(\bar{u}^{(\mp)}, l - r\cos(\alpha_x), m - r\cos(\alpha_y)),
\end{align*}
and 
\[
	\bar{u}^{(\mp)} = \bar{u}^{(\mp)}(l - r\cos(\alpha_x), m - r\cos(\alpha_y), h(l,m,t) - r\cos(\alpha_z)).
\]
Functions $\cos(\alpha_x),\, \cos(\alpha_y)$ and $\cos(\alpha_z)$ are defined in \eqref{sec1.Angles}.

According to the general scheme, substituting \eqref{sec1.regSer} and \eqref{sec1.trLayerSer} into equations \eqref{sec4.TranLayerEq}, and expanding functions in the right part of \eqref{sec4.TranLayerEq} into Taylor's series with respect to $\mu$, allow us to obtain a set of equations for each individual $Q_i^{(\mp)}(\xi,l,m,h,t)$, $i=0,1,\ldots$, from \eqref{sec1.trLayerSer}.

Due to the nature of the inner layer functions, we need to introduce additional conditions about their behavior at infinities
\begin{equation}\label{sec4.InfCond}
	Q_i^{(\mp)}(\mp\infty,l,m,h,t) = 0
\end{equation}
and also to ensure that the equalities from \eqref{sec1.UConnect} are satisfied. Therefore, we rewrite \eqref{sec1.UConnect} with respect to \eqref{sec3.serU}:
\begin{multline}\label{sec4.ConnCondU}
    \varphi^{*}(l,m,h) = \bar{u}_0^{(-)}(l,m,h) + \mu\bar{u}_1^{(-)}(l,m,h) + \cdots \\
    + Q_0^{(-)}(0,l,m,h,t) + \mu Q_1^{(-)}(0,l,m,h,t) + \cdots \\ 
    = \bar{u}_0^{(+)}(l,m,h) + \mu\bar{u}_1^{(+)}(l,m,h) + \cdots \\
    + Q_0^{(+)}(0,l,m,h,t) + \mu Q_1^{(+)}(0,l,m,h,t) + \cdots\;.
\end{multline}

\subsubsection{Zero-order inner layer functions} 
\label{ZeroOrderInnerFunctions}
If we take $\mu^{-1}$ coefficients from \eqref{sec4.TranLayerEq}, $\mu^{0}$ coefficients from \eqref{sec4.ConnCondU}, and condition \eqref{sec4.InfCond}, then we can formulate the following problem
\begin{equation}\label{sec4.TrLayerQ0}
	\begin{cases}
		\displaystyle\pdv[2]{Q_0^{(\mp)}}{\xi} + \alpha_z((h_t + A(h)h_x + B(h)h_y + \varphi^{(\mp)} + Q_0^{(\mp)}))\pdv{Q_0^{(\mp)}}{\xi} = 0, \\[2ex]
		\varphi^{(\mp)}(l,m,h) + Q_0^{(\mp)}(0,l,m,h,t) = \varphi^{*}(l,m,h), \quad Q_0^{(\mp)}(\mp\infty,l,m,h,t) = 0.
	\end{cases}
\end{equation}

For simplicity, we introduce the following notation:
\begin{equation}\label{sec4.tildUReplace}
	\tilde{u}(\xi,h) = 
	\begin{cases}
		\varphi^{(-)}(l,m,h) + Q_0^{(-)}(\xi,l,m,h,t), \quad \xi \leqslant 0, \\
		\varphi^{(+)}(l,m,h) + Q_0^{(+)}(\xi,l,m,h,t), \quad \xi \geqslant 0,
	\end{cases}
\end{equation}
Function $F(\tilde{u}(\xi,h),l,m,h)$ we will write as $\tilde{F}(\xi,l,m,t)$

Clearly, by inserting the substitution \eqref{sec4.tildUReplace} into \eqref{sec4.TrLayerQ0},  we obtain equations that are identical to those in \eqref{sec1.AuxEq}. As we have already shown,  there exist derivatives 
\begin{equation}\label{sec4.DerTildU}
	\begin{aligned}
		&\Phi^{(-)}(\xi,h) := \Phi^{(-)}(\tilde{u}(\xi,h),h) = \pdv{\tilde{u}}{\xi}, \quad \xi\leqslant 0, \\
		&\Phi^{(+)}(\xi,h) := \Phi^{(+)}(\tilde{u}(\xi,h),h) = \pdv{\tilde{u}}{\xi}, \quad \xi\geqslant 0,
	\end{aligned}
\end{equation}
which satisfy equations \eqref{sec1.PhiExp}. 

For simplicity, we introduce the following replacement:
\[
	P^{(\mp)}(h) := h_t + A(h)h_x + B(h)h_y + \varphi^{(\mp)}(x,h).
\]
For any point $(x,y,t)\in \mathbb{R}\times [0,T]$ and function $h(x,y,t)$, satisfying the inequality $P^{(-)}(h)>0$, there exists a solution $Q_0^{(-)}$ for \eqref{sec4.TrLayerQ0} when $\xi\leqslant 0$. Similarly, if it is inequality $P^{(+)}(h)<0$, then there exists a solution $Q_0^{(+)}$ for \eqref{sec4.TrLayerQ0} when $\xi\geqslant 0$.

Thus, it is possible to define  
\begin{equation}\label{sec4.Q0Expres}
	Q_0^{(+)} = \dfrac{2\tilde{P}^{(\mp)}P^{(\mp)}}{e^{\alpha_zP^{(\mp)}\xi} - \tilde{P}^{(\mp)}} .
\end{equation}
Here 
\[
	\tilde{P}^{(\mp)} = \dfrac{\varphi^{*}(x,h) - \varphi^{(\mp)}(x,h)}{2P^{(\mp)} - (\varphi^{*}(x,h) - \varphi^{(\mp)}(x,h))}
\]
and $\varphi^{*}(x,h)$ was specified in the form \eqref{sec1.halfSum}.

\subsubsection{First order and higher inner layer functions}
If we equate elements to the $\mu^0$ power in \eqref{sec4.TranLayerEq}, then we obtain 
\begin{multline}\label{sec4.TrLayerQ1}
    \displaystyle\pdv[2]{Q_1^{(\mp)}}{\xi} + \alpha_z(h_t + A(h)h_x + B(h)h_y + \tilde{u}(\xi,h))\pdv{Q_1^{(\mp)}}{\xi}  \\
    + \alpha_z\Phi^{(\mp)}(\xi,h)Q_1^{(\mp)} = f_1^{(\mp)}(\xi,l,m,t),	
\end{multline}
where 
\begin{align*}
	f_1^{(\mp)} &= \alpha_z^3((h_x^2+1)h_{yy} + (h_y^2+1)h_{xx} - 2h_x h_y h_{xy})\Phi^{(\mp)} \\
	& \quad + \xi\alpha_z^2\qty[\pdv{A}{x}h_x^2 + h_x h_y\qty(\pdv{A}{y} + \pdv{B}{x}) + \pdv{B}{y} - h_x\pdv{A}{z} - h_y\pdv{B}{z}]\Phi^{(\mp)} \\
	& \quad + \xi\alpha_z^2\qty[\pdv{\varphi}{x}h_x + \pdv{\varphi}{y}h_y - \pdv{\varphi}{z}]\Phi^{(\mp)} - \alpha_z\bar{u}_1^{(\mp)}\Phi^{(\mp)} \\
	& \quad + \alpha_z^2\qty[(h_y^2+1)A(h) - h_x\qty(h_t + h_y B(h) + \tilde{u}_0^{(\mp)})]\pdv{Q_0^{(\mp)}}{l} \\
	& \quad + \alpha_z^2\qty[(h_x^2+1)B(h) - h_y\qty(h_t + h_x A(h) + \tilde{u}_0^{(\mp)})]\pdv{Q_0^{(\mp)}}{m} \\
	& \quad + \pdv{Q_0^{(\mp)}}{t} + A(h)\pdv{\varphi}{x} + B(h)\pdv{\varphi}{y} - \tilde{u}_0^{(\mp)}\pdv{\varphi}{z} \\
	& \quad +  \tilde{F}(\xi,l,m,t) - F(\varphi^{(\mp)},l,m,h).
\end{align*}
It is solved for the next boundary conditions $\bar{u}_1^{(\mp)}(l,m,h) + Q_1^{(\mp)}(0,l,m,h,t) = 0$ and $Q_1^{(\mp)}(\mp\infty,l,m,h,t) = 0$.
The solution for such a problem can be written as
\begin{align*}
	Q_1^{(\mp)}(\xi,l,m,h,t) &= -\bar{u}_1^{(\mp)}(l,m,h) \dfrac{\Phi^{(\mp)}(\xi,h)}{\Phi^{(\mp)}(0,h)} \\
	& \quad + \Phi^{(\mp)}(\xi,h) \int_{0}^{\xi} \dfrac{1}{\Phi^{(\mp)}(s_1,h)}\int_{\mp\infty}^{s_1} f_1^{(\mp)}(s_2,l,m,t) \dd{s_2}  \dd{s_1}.
\end{align*}

The rest of the elements with the indexes $k = 2,3\ldots$ can be derived from the problems
\begin{multline*}
	 \pdv[2]{Q_k^{(\mp)}}{\xi} + \alpha_z(h_t + A(h)h_x + B(h)h_y + \tilde{u}(\xi,h))\pdv{Q_k^{(\mp)}}{\xi} \\ 
    + \alpha_z\Phi^{(\mp)}(\xi,h)Q_k^{(\mp)} = f_k^{(\mp)}(\xi,l,m,t),
\end{multline*}
in a similar manner as for $k=1$.

\subsection{Asymptotic representation of the function $h$}\label{subsecFuncH}
To define all the coefficients $h_i(l,m,t),\, i=0,1,\ldots$, we use \eqref{sec1.hSer} and \eqref{sec1.DerUConnect}. However, we firstly need to define the derivative $\pdv{n}$ in the direction normal to the surface $h(x,y,t)$:
\[
	\pdv{n} = (n, \bigtriangledown) = \pdv{r} = -\alpha_x\pdv{x} - \alpha_y\pdv{y} + \alpha_z\pdv{z},
\]
where $\alpha_x,\, \alpha_y$, and $\alpha_z$ are defined in \eqref{sec1.Angles}. In terms of the stretched variable $\xi$, this derivative transforms to 
\[
	\pdv{n} = \dfrac{1}{\mu}\pdv{\xi}.
\]
Now, by combining all the above with asymptotic representations \eqref{sec1.regSer} and \eqref{sec1.trLayerSer}, we can rewrite the matching condition \eqref{sec1.DerUConnect} in the form:
\begin{equation}\label{sec5.matchCond}
	\begin{aligned} 
		&-\alpha_x\pdv{\varphi^{(-)}}{x}\qty(l,m,h) - \alpha_y\pdv{\varphi^{(-)}}{y}\qty(l,m,h) + \alpha_z\pdv{\varphi^{(-)}}{z}\qty(l,m,h) \\
		& \qquad -\mu \alpha_x\pdv{\bar{u}_1^{(-)}}{x}\qty(l,m,h) - \mu \alpha_y\pdv{\bar{u}_1^{(-)}}{y}\qty(l,m,h) + \mu \alpha_z\pdv{\bar{u}_1^{(-)}}{z}\qty(l,m,h) + \cdots \\
		& \qquad + \dfrac{1}{\mu}\pdv{Q_0^{(-)}}{\xi}\qty(0,l,m,h,t) + \pdv{Q_1^{(-)}}{\xi}\qty(0,l,m,h,t) + \cdots \\
		&= -\alpha_x\pdv{\varphi^{(+)}}{x}\qty(l,m,h) - \alpha_y\pdv{\varphi^{(+)}}{y}\qty(l,m,h) + \alpha_z\pdv{\varphi^{(+)}}{z}\qty(l,m,h) \\
		& \qquad -\mu \alpha_x\pdv{\bar{u}_1^{(+)}}{x}\qty(l,m,h) - \mu \alpha_y\pdv{\bar{u}_1^{(+)}}{y}\qty(l,m,h) + \mu \alpha_z\pdv{\bar{u}_1^{(+)}}{z}\qty(l,m,h) + \cdots\\
		& \qquad  + \dfrac{1}{\mu}\pdv{Q_0^{(+)}}{\xi}\qty(0,l,m,h,t) + \pdv{Q_1^{(+)}}{\xi}\qty(0,l,m,h,t) + \cdots\;.
	\end{aligned}
\end{equation}

For the sake of simplicity, we introduce a new function
\[
	H(l,m,h,t,\mu) := \mu\pdv{U^{(-)}}{n}\qty(l,m,h,t,\mu) - \mu\pdv{U^{(+)}}{n}\qty(l,m,h,t,\mu).
\]
Therefore, instead of \eqref{sec1.DerUConnect}, we write the following equation 
\[
	H(l,m,h,t,\mu) = H_0(l,m,h,t) + \mu H_1(l,m,h,t) +  \ldots = 0,
\]
where 
\begin{align} 
	H_0 &= \pdv{Q_0^{(-)}}{\xi}\qty(0,l,m,h,t) - \pdv{Q_0^{(+)}}{\xi}\qty(0,l,m,h,t), \label{equationH0} \\
	H_1 &= \pdv{Q_1^{(-)}}{\xi}\qty(0,l,m,h,t)  - \pdv{Q_1^{(+)}}{\xi}\qty(0,l,m,h,t) \notag  \\
	&- \alpha_x\pdv{\varphi^{(-)}}{x}\qty(l,m,h)- \alpha_y\pdv{\varphi^{(-)}}{y}\qty(l,m,h) + \alpha_z\pdv{\varphi^{(-)}}{z}\qty(l,m,h) \notag   \\
	&+ \alpha_x\pdv{\varphi^{(+)}}{x}\qty(l,m,h) +\alpha_y\pdv{\varphi^{(+)}}{y}\qty(l,m,h) - \alpha_z\pdv{\varphi^{(+)}}{z}\qty(l,m,h). \label{equationH1}
\end{align}

As a result, the smooth matching condition is obtained by equating coefficients at $\mu^0$ and has the following representation: 
\[
	H_0(l,m,h(l,m,t),t) = \Phi^{(-)}(\varphi^{*}(l,m,t),h) - \Phi^{(+)}(\varphi^{*}(l,m,t),h) = 0.
\]

Thus, due to the fact that previous equation has the same form as in \eqref{sec1.PhiLEqualPhiR}, we can use Condition \eqref{C4} to define $h_0(x,y,t)$ from \eqref{sec1.hSer}.

The smooth matching condition of order $\mu^1$ can be written as:
\begin{equation}\label{sec5.matchC2}
	\begin{aligned} 
		\alpha_z \pdv{h_1}{t} &+ \pdv{H_0}{h_x}\qty(x, y, h_0, t)\pdv{h_1}{x} + \pdv{H_0}{h_y}\qty(x, y, h_0, t)\pdv{h_1}{y} \\
		&+ \pdv{H_0}{h}\qty(x, y, h_0, t)h_1 + H_1(l,m,h_0,t) = 0.
	\end{aligned}
\end{equation}
The functions $H_0$ and $H_1$ are defined in \eqref{equationH0} and \eqref{equationH1}, respectively. We can define $h_1(x,y,t)$ as a solution of \eqref{sec5.matchC2} with boundary conditions
\[
h_1(x+L,y,t) = h_1(x,y,t) = h_1(x,y+M,t), \quad h_1(x,y,0) = 0.
\]

This problem is solvable due to the positive coefficient in front of $\displaystyle\pdv{h_1}{t}$.  We can use the same approach to find coefficients $h_k(x,y,t),\, k=2,3,\ldots$ from the problems:
\begin{align*} 
	\alpha_z \pdv{h_k}{t} &+ \pdv{H_0}{h_x}\qty(x, y, h_0, t)\pdv{h_k}{x} + \pdv{H_0}{h_y}\qty(x, y, h_0, t)\pdv{h_k}{y} \\
	&+ \pdv{H_0}{h}\qty(x, y, h_0, t)h_k + H_k(l,m,h_0,t) = 0, 
\end{align*}
with conditions
\[
h_k(x + L,y,t) = h_k(x,y,t) = h_k(x,y + M,t), \quad h_k(x,y,0) = 0.
\]

\section{Main results}\label{SecMainResults}
According to the asymptotic analysis in Section \ref{section3}, we can now formulate the main theorem of this article as follows:

\begin{theorem}\label{th1}
	If conditions \eqref{C1}, \eqref{C2}, \eqref{C3}, and \eqref{C4} are met, then the system \eqref{sec1.InitProb} has a solution $u(x,y,z,t,\mu)$, which is $L$-periodic in the variable $x$ and $M$-periodic in the variable $y$. Additionally, this solution belongs to the class
    \[
    C(\bar{D}\times[0,T])\cap C^{2,2,1,1}(\bar{D}\times[0,T])\cap C^{2,2,2,1}((\bar{D}^{(-)}\cup \bar{D}^{(+)})\times[0,T]),
    \]
    such that it lies between the upper and lower solutions for $t\in [0,T]$, and the function $U_n(x,y,z,t,\mu)$ is its uniform asymptotic approximation in the domain $\bar{D}\times [0,T]$ with the accuracy $O(\mu^{n+1})$. In other words, in the area $\bar{D}\times [0,T]$, we have the following estimate:
	\[
		\abs{u(x,y,z,t,\mu) - U_n(x,y,z,t,\mu)} \leqslant C\mu^{n+1},
	\] 
	where $C$ is some positive constant. The function $U_n$ has the following representation:
\begin{equation}\label{sec7.finSol}
    U_n= 
    \begin{cases}
        \bar{u}^{(-)}(x,y,z,\mu) + Q^{(-)}(\xi,l,m,h(l,m,t),t,\mu), \\
        \bar{u}^{(+)}(x,y,z,\mu) + Q^{(+)}(\xi,l,m,h(l,m,t),t,\mu).
    \end{cases}
\end{equation}
 Here,  $\bar{u}^{(\mp)}$ and $Q^{(\mp)}$ are defined in \eqref{sec1.regSer} and \eqref{sec1.trLayerSer} respectively.

\end{theorem}
\begin{proof}
    The proof is based on the classical method of differential inequalities, as detailed in \cite{Nef1995,AntLevNef2014,ChaZha2022}. This approach dictates that a solution for the initial problem \eqref{sec1.InitProb} exists if there are upper and lower solutions for the same problem \eqref{sec1.InitProb}, and they satisfy the following properties: 
    \begin{condition}\label{LowUpCond1}
        The inequalities
        \[
            \displaystyle L[\beta] \leqslant 0 \leqslant L[\alpha], \\
        \]     
        \[     
             \alpha \leqslant \beta, \quad \text{for all }(x,y,z,t) \in \bar{D} \times [0,T],
        \]
        \begin{align*}
            &\alpha(x,y,z,t,\mu) = \alpha(x+L,y,z,t,\mu), \quad \alpha(x,y,z,t,\mu) = \alpha(x,y+M,z,t,\mu), \\
            &\beta(x,y,z,t,\mu) = \beta(x+L,y,z,t,\mu), \quad \beta(x,y,z,t,\mu) = \beta(x,y+M,z,t,\mu)
        \end{align*}
        are satisfied for smooth functions $\alpha(x,y,z,t,\mu)$ and $\beta(x,y,z,t,\mu)$
        with differential operators
        \[
            \displaystyle L[\gamma] := \mu\Delta \gamma - \pdv{\gamma}{t} = A\pdv{\gamma}{x} + B\pdv{\gamma}{y} - \gamma \pdv{\gamma}{z} + F(\gamma,x,y,z), \quad \gamma = \alpha, \beta.
        \]
    \end{condition}
    \begin{condition}\label{LowUpCond2}
        The relations of the upper and lower solutions to the boundary and initial conditions for any $(x,y,z) \in \bar{D}, \mu \in (0, \mu_0]$:
        \begin{align*}
             &\alpha(x,y,0,t,\mu) \leqslant u^{0}(x,y) \leqslant \beta(x,y,0,t,\mu), \\ 
             &\alpha(x,y,a,t,\mu) \leqslant u^{a}(x,y) \leqslant \beta(x,y,a,t,\mu), \\ 
             &\alpha(x,y,z,0,\mu) \leqslant u_{init}(x,y,z,\mu) \leqslant \beta(x,y,z,0,\mu),
        \end{align*}
    \end{condition}
    \begin{condition}\label{LowUpCond3}
         Regarding the derivative discontinuities for lower and upper solutions along the normals to the curves $h_{\alpha, \beta}$, it is established that:
        \begin{align*}
             &\pdv{\alpha}{n} \qty(x, y, h_{\alpha}(x,y,t)+0, t, \mu) - \pdv{\alpha}{n} \qty(x, y, h_{\alpha}(x,y,t)-0, t, \mu) \geqslant 0, \\
             &\pdv{\beta}{n} \qty(x, y, h_{\beta}(x,y,t)+0, t, \mu) - \pdv{\beta}{n} \qty(x, y, h_{\beta}(x,y,t)-0, t, \mu) \geqslant 0, 
        \end{align*}
        where $h_{\alpha}(x,y,t)$ and $h_{\beta}(x,y,t)$ are the curves where the upper and lower solutions are not smooth.
    \end{condition}

    A detailed examination of these conditions, as well as the construction of upper and lower solutions, is carried out similarly to the method presented in \cite{AntVolLevNef2017}. The transition functions need to be redefined as follows:
    \begin{align*} 
        h_{\alpha}(x,y,t) = \sum_{i=0}^{n} \mu^{i} h_i(x,y,t) + \mu^{n+1}\delta_{h}(x,y,t), \\
        h_{\beta}(x,y,t) = \sum_{i=0}^{n} \mu^{i} h_i(x,y,t) - \mu^{n+1}\delta_{h}(x,y,t),
    \end{align*}
    where $\delta_{h}(x,y,t)$ is chosen according to Condition \ref{LowUpCond3}. Subsequently, $\alpha(x,y,$ $z,t,\mu)$ and $\beta(x,y,z,t,\mu)$ can be constructed by modifications of the asymptotic series \eqref{sec3.serU} as follows:
    \begin{align}\label{lowerSolExp}
        \alpha^{(\mp)} = U_{n+1}^{(\mp)}\Big|_{\xi_{\alpha}} - \mu^{n+1} \qty(\delta^{(\mp)} + q^{(\mp)}(\xi_\alpha,t)), \qquad (x,y,z,t) \in \bar{D}_{\alpha} \times [0,T],\\
        \beta^{(\mp)} = U_{n+1}^{(\mp)}\Big|_{\xi_{\beta}} + \mu^{n+1} \qty(\delta^{(\mp)} + q^{(\mp)}(\xi_\beta,t)), \qquad (x,y,z,t) \in \bar{D}_{\beta} \times [0,T].
    \end{align}
     Here, $\bar{D}_{\alpha}$ and $\bar{D}_{\beta}$ are defined similarly to $\bar{D}$, but specifically for the discontinuous curves $h_{\alpha}$ and $h_{\beta}$. Using the sewing technique, it is possible to establish all coefficients of the representations of the upper and lower functions. More details about their construction are presented for the two-dimensional case in \cite{AntVolLevNef2017}.
\end{proof}
Therefore, we can determine all elements of \eqref{sec1.regSer} and \eqref{sec1.trLayerSer} up to the $k$-th order, which depends on our desired accuracy.

\section{A Numerical example}\label{secNumExample}

We analyze a three-dimensional reaction-diffusion-advection differential equation

\begin{equation} \label{forwardexample1}
    \displaystyle \mu\Delta u -  \frac{\partial u}{\partial t} =  \sin(\pi x)  \frac{\partial u}{\partial x} + \cos(\frac{\pi x}{4})  \frac{\partial u}{\partial y}-u \frac{\partial u}{\partial z}  +f(x,y,z),
\end{equation}
where $x, y, z \in \Omega = [-1,1] \times [-1,1]\times [0,1]$, $t \in [0,0.85]$, $\mu = 0.01$, and the source function $f(x,y,z)=\cos{ ( \frac{\pi x}{4}  )} \cos{( \frac{\pi y}{4}  )} \cos{( \frac{\pi z}{4}  )}$. For this particular case, we employ boundary conditions $u(x,y,0,t)=-6, \, u(x,y,1,t)=4, \, u(-1,y,z,t) = u(1,y,z,t), \, u(x,-1,z,t) = u(x,1,z,t)$ and initial condition $\displaystyle  u(x,y,z,0)=u_{init}(x,y,z,\mu )$. Employing the method of asymptotic expansions, the original equation is decomposed into two sub-problems to identify the outer functions:

\begin{equation*}
    \displaystyle  \sin(\pi x)   \frac{\partial \varphi^{(\mp)} }{\partial x} + \cos(\frac{\pi x}{4})  \frac{\partial \varphi^{(\mp)} }{\partial y}-\varphi^{(\mp)}  \frac{\partial \varphi^{(\mp)}}{\partial z}  = \cos{ ( \frac{\pi x}{4}  )} \cos{( \frac{\pi y}{4}  )} \cos{( \frac{\pi z}{4}  )}
\end{equation*}
with boundary conditions 
\[
    \varphi^{(-)}(x,y,0)=-6, \quad \varphi^{(+)}(x,y,1)=4,
\]
and periodic conditions
\[
     \varphi^{(\mp)}(-1,y,z) = \varphi^{(\mp)}(1,y,z), \, \varphi^{(\mp)}(x,-1,z) = \varphi^{(\mp)}(x,1,z).
\]
The leading term of the propagating front $h_0(x,y,t)$ could be established as 
\begin{equation}\label{eq44}
\begin{cases}
    \displaystyle {h_0}_{t} +\sin(\pi x){h_0}_{x}+\cos(\frac{\pi x}{4}){h_0}_{y} =-\frac{1}{2}  \left( \varphi^{(-)} +\varphi^{(+)} \right),\\
    h_0(x,y,0)=h_0^{*}=0, \\ 
    h_0(-1,y,0)=h_0(1,y,0), \quad h_0(x,-1,0)=h_0(x,1,0).
\end{cases}
\end{equation}

\begin{figure}[ht] 
\begin{center}
\includegraphics[width=0.4\linewidth,height=0.4\textwidth,keepaspectratio]{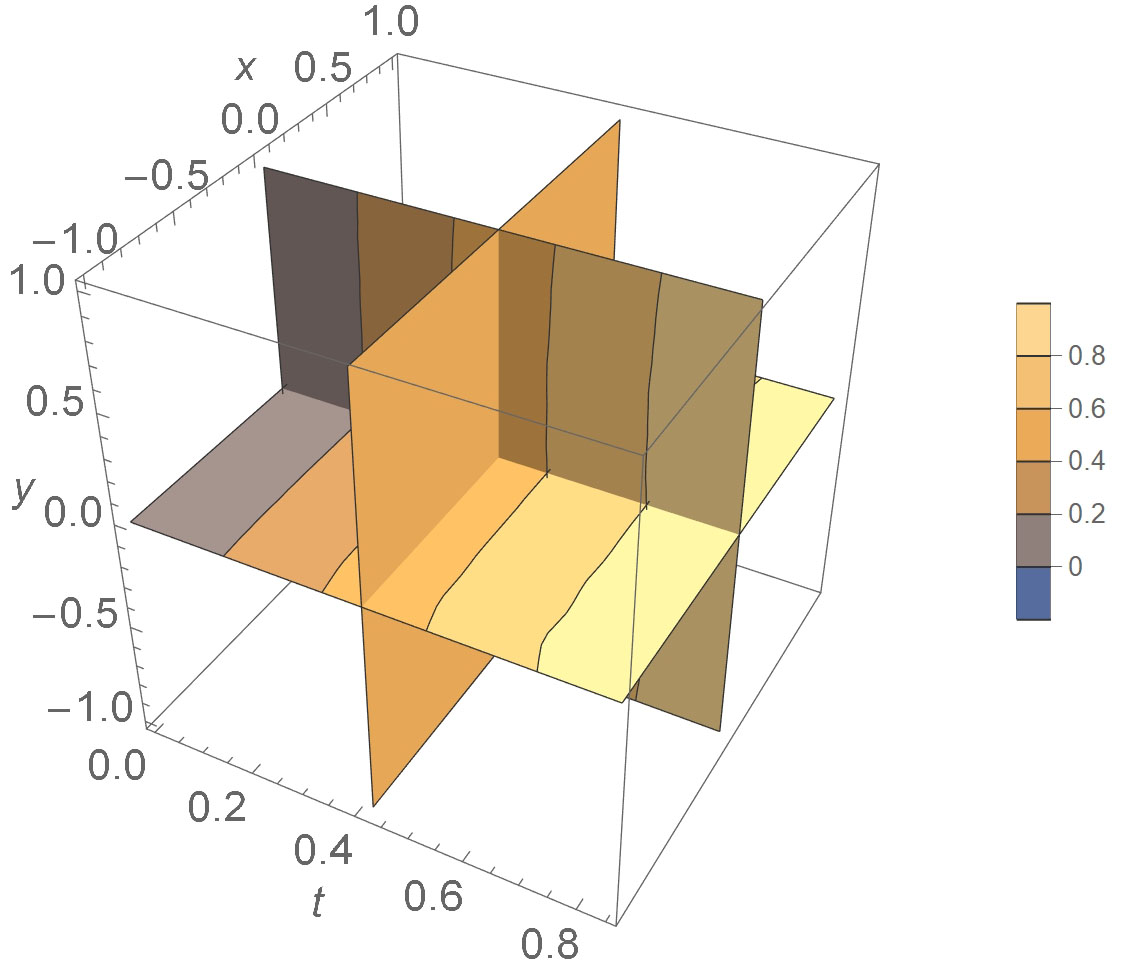}
\caption{Solution of \eqref{eq44} (obtained numerically).}
\label{fig:x0example1}
\end{center}
\end{figure}

The analysis of the equation \eqref{eq44}, illustrated in Fig. \ref{fig:x0example1}, reveals the presence of a transition layer confined to the domain $0 \leq h_0(x,y,t) \leq 1$ across all values of $x$ and $y$ within the interval $[-1,1]$ and for time $t$ ranging from 0 to 0.85. This finding provides numerical validation for the Assumption \ref{C4}

To solve the equation \eqref{forwardexample1} numerically, it is important to establish an initial condition, which can be formulated as follows:
\begin{equation*}
    \displaystyle u_{init} =\frac{u^{a}}{2}\qty(1 + \Theta) + \frac{u^{-a}}{2} \qty(1 - \Theta).
\end{equation*}
where $\Theta = \tanh(x+y+\frac{z-h_0(x,y,0)}{0.1\mu})$. In this example, it is reduced to
\begin{equation*}
    \displaystyle u_{init} = 5\tanh\left(x+y+\frac{z}{0.001}\right)-1,
\end{equation*}
setting the initial location of the transition layer at $h_0(x,y,0)=0$.

Our hypotheses are confirmed through this study, showing that equation \eqref{forwardexample1} admits an autowave solution featuring a moving transitional layer close to $h_0(x,y,t)$. In the leading-order approximation, the solution is expressed as:

\begin{align} \label{asymptoticsolEXAMPLE1}
U_{0}=
\begin{cases}
\displaystyle \varphi^{(-)}(x,y,z)  -\frac{\varphi^{(-)}(x,y,h_0)-\varphi^{(+)}(x,y,h_0)}{\exp \left(-\Phi_{*}(h_0)  \right)+1} ,  \quad z \in [0,h_0], \\
\displaystyle \varphi^{(+)}(x,y,z)+\frac{\varphi^{(-)}(x,y,h_0)-\varphi^{(+)}(x,y,h_0)}{\exp \left(\Phi_{*}(h_0) \right)+1} , \quad z \in [h_0,1].
\end{cases}
\end{align}
Here, the term $\Phi_{*}(h_0)$ is defined as:
\[
\Phi_{*}(h_0) = \frac{ \left(h_0-z   \right) \left(\varphi^{(-)}(x,y,h_0)-\varphi^{(+)}(x,y,h_0) \right) \left(1-{h_0}_{x}-{h_0}_{y} \right)}{2 \mu}
\]

To validate our methodology, we correspond the asymptotic solution \eqref{asymptoticsolEXAMPLE1} with a numerical solution at two specific instances, $t_1 = 0.2$ and $t_2 = 0.6$, considering the numerical solution as the benchmark. These comparisons are showcased in  Figures \ref{fig:asymptSOL} and \ref{fig:numSOL} for the asymptotic solution and Figures \ref{fig:asymptSOL2} and \ref{fig:numSOL2} for the numerical solution, at times $t_1 = 0.2$ and $t_2 = 0.6$, respectively.

The effectiveness of the asymptotic solution is evidenced by a small relative error. Specifically, the relative error at $t_1 = 0.2$ is computed as:
$$ 
    \frac{\| U_0(x,y,z,t_1) - u(x,y,z,t_1) \|_{L^{2}(\Omega)}}{\|u(x,y,z,t_1) \|_{L^{2}(\Omega)}} = 0.179. 
$$

Similarly, at $t_2 = 0.6$, the value of the relative error is:
$$
    \frac{\| U_0(x,y,z,t_2) - u(x,y,z,t_2) \|_{L^{2}(\Omega)}}{\|u(x,y,z,t_2) \|_{L^{2}(\Omega)}} = 0.194. 
$$

\begin{figure}[ht!]
 \centering
\subfigure[]{
\includegraphics[width=0.4\linewidth]{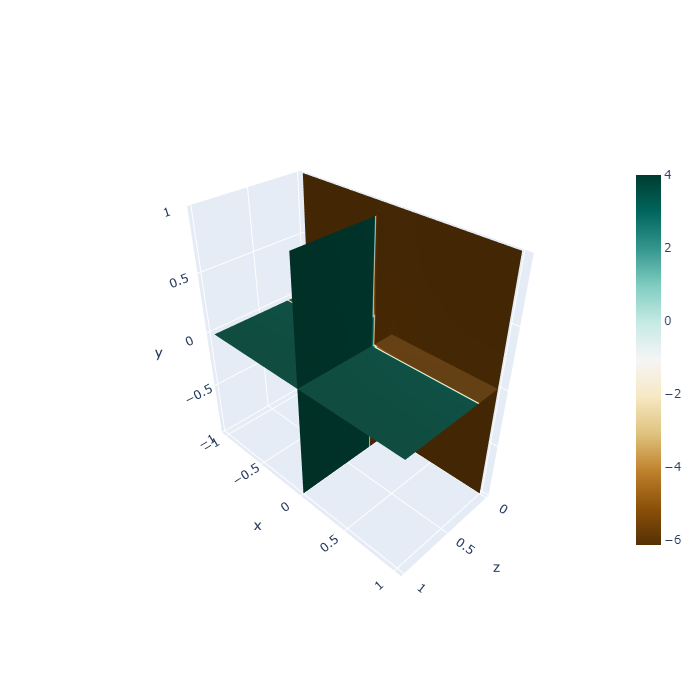} \label{fig:asymptSOL} }
\hspace{0ex}
\subfigure[]{
\includegraphics[width=0.4\linewidth]{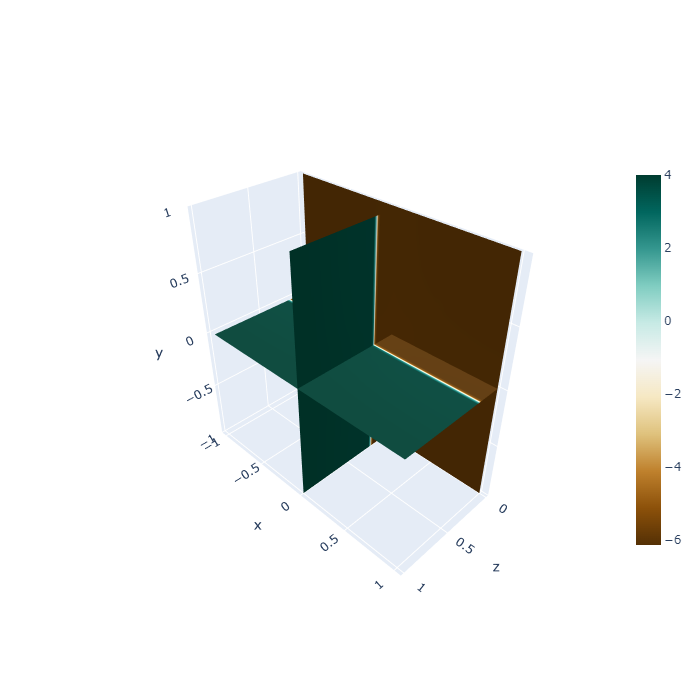} \label{fig:numSOL} }
\caption{Comparison of the asymptotic solution \subref{fig:asymptSOL}   and numerical solution \subref{fig:numSOL} at the moment  $t_1=0.2$. } \label{solutionsexample1}
\end{figure}

\begin{figure}[H]
 \centering
\subfigure[]{
\includegraphics[width=0.4\linewidth]{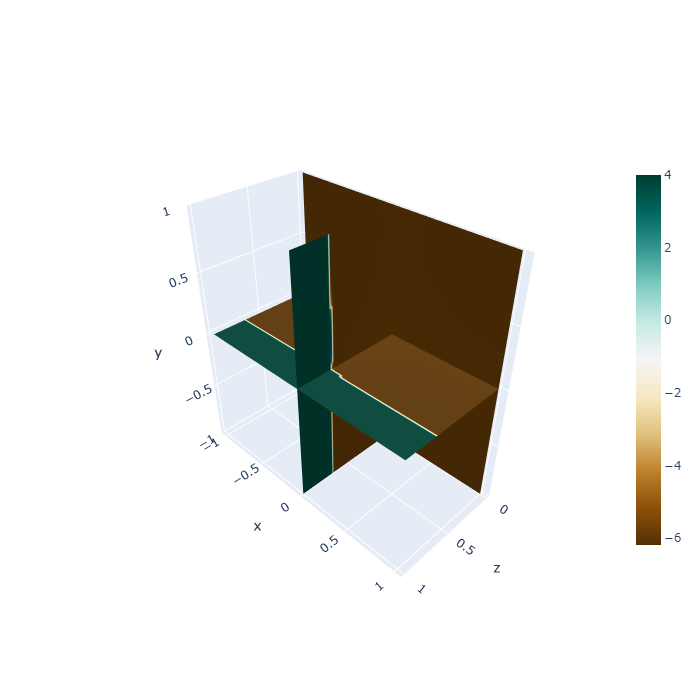} \label{fig:asymptSOL2} }
\hspace{0ex}
\subfigure[]{
\includegraphics[width=0.4\linewidth]{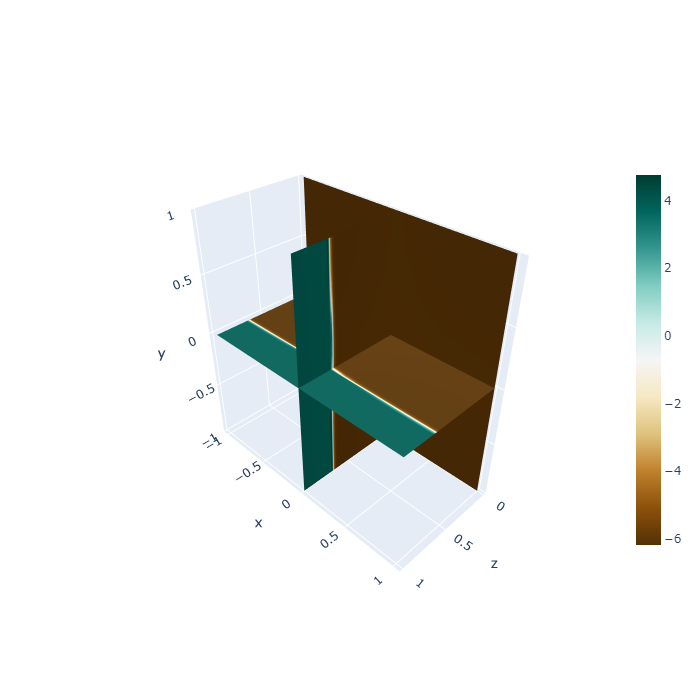} \label{fig:numSOL2} }
\caption{Comparison of the asymptotic solution \subref{fig:asymptSOL2}   and numerical solution \subref{fig:numSOL2} at the moment $t_2=0.6$. } \label{solutionsexample2}
\end{figure}




\section*{Data availability}
No data was used for the research described in the article.

\section*{Funding}
This work was supported by the National Natural Science Foundation of China (Grants 12350410359 \& 12171036), the National Key Research and Development Program of China (Grant 2022YFC3310300) and the Beijing Natural Science Foundation (Grant Z210001).

\section*{Conflict of interest}
The authors of this work declare that they have no conflict of interest.
\smallskip
\begin{center}{REFERENCES}\end{center}{\small
\bibliographystyle{plain} 
\bibliography{reference} 

\begin{thebibliography}{10}

\bibitem{VolNef2006}
V.~Volkov and N.~Nefedov, ``Development of the asymptotic method of
  differential inequalities for investigation of periodic contrast structures
  in reaction-diffusion equations,'' {\em Zh. Vychisl. Mat. Mat. Fiz.},
  vol.~46, no.~4, pp.~615--623, 2006.

\bibitem{BozNef2010}
Y.~Bozhevol'nov and N.~Nefedov, ``Front motion in the parabolic
  reaction-diffusion problem,'' {\em Zh. Vychisl. Mat. Mat. Fiz.}, vol.~50,
  no.~2, pp.~276--285, 2010.

\bibitem{AntLevNef2014}
E.~Antipov, N.~Levashova, and N.~Nefedov, ``Asymptotics of the front motion in
  the reaction-diffusion-advection problem,'' {\em Comput. Math. and Math.
  Phys.}, vol.~54, no.~10, pp.~1536--1549, 2014.

\bibitem{NefYag2015}
N.~Nefedov and A.~Yagremtsev, ``On extension of asymptotic comparison principle
  for time periodic reaction-diffusion-advection systems with boundary and
  internal layers,'' in {\em Finite Difference Methods,Theory and
  Applications}, (Cham), pp.~62--71, Springer International Publishing, 2015.

\bibitem{VolNefAnt2015}
V.~Volkov, N.~Nefedov, and E.~Antipov, ``Asymptotic-numerical method for moving
  fronts in two-dimensional r-d-a problems,'' in {\em Finite Difference
  Methods,Theory and Applications}, (Cham), pp.~408--416, Springer
  International Publishing, 2015.

\bibitem{LibTyaPei2003}
M.~Liberman, M.~Ivanov, O.~Peil, D.~Valiev, and L.-E. Eriksson, ``Numerical
  studies of curved stationary flames in wide tubes,'' {\em Combust. Theor.
  Model.}, vol.~7, no.~4, pp.~653--676, 2003.

\bibitem{Rud2017}
O.~Rudenko, ``Inhomogeneous burgers equation with modular nonlinearity:
  Excitation and evolution of high-intensity waves,'' {\em Dokl. Math.},
  vol.~95, no.~3, pp.~291--294, 2017.

\bibitem{Nef2019}
N.~Nefedov, ``The existence and asymptotic stability of periodic solutions with
  an interior layer of burgers type equations with modular advection,'' {\em
  Math. Model. Nat. Phenom.}, vol.~14, no.~4, p.~401, 2019.

\bibitem{LukVolNef2017}
D.~Luk'yanenko, V.~Volkov, and N.~Nefedov, ``Dynamically adapted mesh
  construction for the efficient numerical solution of a singular perturbed
  reaction-diffusion-advection equation,'' {\em Model. Anal. Inform. Sist.},
  vol.~24, no.~3, pp.~322--338, 2017.

\bibitem{LukShiVol2018}
D.~Lukyanenko, M.~Shishlenin, and V.~Volkov, ``Solving of the coefficient
  inverse problems for a nonlinear singularly perturbed
  reaction-diffusion-advection equation with the final time data,'' {\em
  Commun. Nonlinear Sci. Numer. Simul.}, vol.~54, pp.~233--247, 2018.

\bibitem{ChaZha2022}
D.~Chaikovskii and Y.~Zhang, ``Convergence analysis for forward and inverse
  problems in singularly perturbed time-dependent reaction-advection-diffusion
  equations,'' {\em J. Comput. Phys.}, vol.~470, p.~111609, 2022.

\bibitem{ChaLiuZha2023}
D.~Chaikovskii, A.~Liubavin, and Y.~Zhang, ``Asymptotic expansion
  regularization for inverse source problems in two-dimensional singularly
  perturbed nonlinear parabolic pdes,'' {\em CSIAM Trans. Appl. Math.}, vol.~4,
  no.~4, pp.~721--757, 2023.

\bibitem{ChaiYe2023}
D.~Chaikovskii and Y.~Zhang, ``{Solving forward and inverse problems involving
  a nonlinear three-dimensional partial differential equation via asymptotic
  expansions},'' {\em IMA J. Appl. Math}, vol.~88, pp.~525--557, 08 2023.

\bibitem{Volk2011}
V.~Volkov, N.~Grachev, A.~Dmitriev, and N.~Nefedov, ``Front formation and
  dynamics in one reaction-diffusion-advection model,'' {\em Math. Models.
  Comput. Simul.}, vol.~3, no.~2, pp.~158--164, 2011.

\bibitem{Eva2010}
L.~Evans, {\em Partial Differential Equations}.
\newblock Graduate studies in mathematics, American Mathematical Society, 2010.

\bibitem{Nef1995}
N.~Nefedov, ``The method of differential inequalities for some classes of
  nonlinear singularly perturbed problems with internal layers,'' {\em Differ.
  Uravn.}, vol.~31, no.~7, p.~1, 1995.

\bibitem{AntVolLevNef2017}
E.~Antipov, V.~Volkov, N.~Levashova, and N.~Nefedov, ``Moving front solution of
  the reaction-diffusion problem,'' {\em Model. Anal. Inform. Sist.}, vol.~24,
  no.~3, pp.~259--279, 2017.

\bibitem{debnath1997nonlinear}
L.~Debnath, {\em Nonlinear Partial Differential Equations for Scientists and
  Engineers}.
\newblock Birkh{\"a}user, {2}~ed., 2005.

\end{thebibliography}
}


\section{Appendix: The well-posedness of conditions \eqref{C2} and \eqref{C4}.}
\label{Appendix}

In this appendix, by using standard arguments of partial differential equations (cf. \cite{Eva2010}), we prove the existence of solutions for partial differential equations \eqref{sec1.DegEq} and \eqref{sec1.TransFuncEq} under some mild assumptions. For this purpose, we consider the general first-order quasi-linear PDE:
\begin{equation}\label{sec8.Gen3DQuasiEq}
    a_1(x, y, z, u) u_{x} + a_2(x, y, z, u) u_{y} + a_3(x, y, z, u) u_{z} = f(x, y, z, u).
\end{equation}

The corresponding system of ordinary differential equations (ODEs) for characteristic curves reads:
\begin{equation}\label{sec8.CharODE}
\left\{\begin{array}{l}
	\displaystyle \frac{d x}{d t}=a_1(x, y, z, w), \\
	\displaystyle \frac{d y}{d t}=a_2(x, y, z, w), \\
	\displaystyle \frac{d z}{d t}=a_3(x, y, z, w), \\
	\displaystyle \frac{d w}{d t}=f(x, y, z, w).
\end{array}\right.
\end{equation}

\begin{theorem}\label{characteristictheorem}
(\cite[Theorem 3.5.3]{debnath1997nonlinear},\cite[Section 3]{Eva2010}) 
    Suppose that $x_0(s1,s2)$, $y_0(s_1,s_2)$, $z_0(s_1,s_2)$ and $u_0(s_1,s_2)$ are continuous differentiable functions of $s_1$, $s_2$ in a closed interval $s_1, s_2 \in [0,1]$ and that $a_1$, $a_2$, $a_3$ and $f$ are functions of $x$, $y$, $z$ and $u$ with continuous first order partial derivatives with respect to their arguments in some domain $D$ of $(x,y,z,w)$-space containing the initial curve
    \begin{equation}\label{sec8.CharInitData}
        \Gamma : x=x_0(s_1, s_2), \quad y=y_0(s_1,s_2), \quad z=z_0(s_1,s_2), \quad u=u_0(s_1,s_2),
    \end{equation}
    where $s_1 \in [0,1]$, $s_2 \in [0,1]$, and satisfying the condition
    \begin{equation}\label{sec8.CharCond}
        J=\left.\operatorname{det}\left(\begin{array}{ccc}
	\displaystyle  a_1 &\displaystyle  a_2 &\displaystyle  a_3 \\
	\displaystyle \frac{\partial x_0}{\partial s_1} & \displaystyle            \frac{\partial y_0}{\partial s_1} & \displaystyle \frac{\partial z_0}{\partial s_1} \\
	\displaystyle \frac{\partial x_0}{\partial s_2} & \displaystyle    \frac{\partial y_0}{\partial s_2} & \displaystyle \frac{\partial z_0}{\partial s_2}
    \end{array}\right)\right|_\Gamma \neq 0,
    \end{equation}
    Then, there exists a unique solution $u = u(x,y,z)$ of \eqref{sec8.Gen3DQuasiEq} in the neighbourhood of $C : x=x_0(s_1,s_2), y=y_0(s_1,s_2), z=z_0(s_1,s_2), s_1, s_2 \in [0,1]$, and the solution satisfies the initial condition $u_0(s_1,s_2) = u(x_0, y_0, z_0)$. 
\end{theorem}

Short overview of the proof can be described in the following steps. First, it requires to solve the equation \eqref{sec8.CharODE} with $t=0$ and initial data \eqref{sec8.CharInitData}, which leads to unique solution
\begin{equation}
    x = X(s_1,s_2,t), y = Y(s_1,s_2,t), z = Z(s_1,s_2,t), w = W(s_1,s_2,t), 
\end{equation}
where $s_1, s_2 \in [0,1], \quad t\in [0,T]$ ($T$ is some positive constant), and solutions are satisfying equality
\[
    (X(s_1,s_2,0), Y(s_1,s_2,0), Z(s_1,s_2,0)) = (x_0(s_1,s_2), y_0(s_1,s_2), z_0(s_1,s_2)).
\]
Next, Condition \eqref{sec8.CharCond} allows us to define $(s_1,s_2,t)$ in terms of $(x,y,z)$ as follows
\[
    s_1 = \phi_1(x,y,z), \quad s_2 = \phi_2(x,y,z), \quad t = \psi(x,y,z).
\]
Finally, if we substitute them into $z = Z(s_1,s_2,t)$, then we obtain 
\[
    u = Z(s_1,s_2,t) = Z(\phi_1(x,y,z), \phi_2(x,y,z), \psi(x,y,z)) \equiv u(x,y,z),
\]
which is solution for the problem \eqref{sec8.Gen3DQuasiEq}. More details about solving nonlinear and quasi-linear problems with usage of characteristics can be found in \cite{Eva2010}. 
From the above analysis, we can establish two theorems regarding the existence of solutions for the problems defined in equations \eqref{sec1.DegEq} and \eqref{sec1.TransFuncEq}.

\begin{theorem}\label{Th.DegSolExist}
Suppose that $A(x,y,z)$, $B(x,y,z)$ and $F(u,x,y,z)$ are functions of $x$, $y$, $z$, and $u$ with continuous first-order partial derivatives with respect to their arguments in the domain $\bar{D}$ containing the initial curves
\[
    \Gamma^0 : x=s_1, \quad y=s_2, \quad z=0, \quad u=u^{0}(s_1,s_2),
\]
and
\[
    \Gamma^a : x=s_1, \quad y=s_2, \quad z=a, \quad u=u^{a}(s_1,s_2),
\]
If functions $u^{0}(s_1,s_2)$ and $u^{a}(s_1,s_2)$ are continuously differentiable functions of $s_1$ and $s_2$ in a closed interval $s_1, s_2 \in [0,1]$, and the inequalities
\[
    u^{0}(s_1,s_2) \neq 0, \qquad u^{a}(s_1,s_2) \neq 0
\]
are satisfied, then problem \eqref{sec1.DegEq} has in the domain $\bar{D}$ a unique solution $u=\varphi^{(-)}(x,y,z)$ for the initial curve $\Gamma^0$ and a solution $u=\varphi^{(+)}(x,y,z)$ for the initial curve $\Gamma^a$.
\end{theorem}

\begin{proof}
In the case of the problem, we let $a_1(x_1, x_2, x_3, u) = A(x,y,z)$, $a_2(x_1,$ $  x_2, x_3, u) = B(x,y,z)$, $a_3(x_1, x_2, x_3, u) = u$, $c(x_1, x_2, x_3, u) = -F(u,x,y,z)$. Condition \ref{C2} implies that we need two solutions, so we use two initial curves $\Gamma^0$ and $\Gamma^a$. The characteristic conditions for $\Gamma^{0}$ are:
\[
    J=\left.\operatorname{det}\left(\begin{array}{ccc}
    A & B & u \\
    1 & 0 & 0 \\
    0 & 1 & 0
\end{array}\right)\right|_{\Gamma^{0}} = u^{0}(s_1,s_2) \neq 0,
\]
and for $\Gamma^{a}$:
\[
    J=\left.\operatorname{det}\left(\begin{array}{ccc}
    A & B & u \\
    1 & 0 & 0 \\
    0 & 1 & 0
\end{array}\right)\right|_{\Gamma^{a}} = u^{a}(s_1,s_2) \neq 0,
\]
If the above inequalities are satisfied, then we can apply Theorem \ref{characteristictheorem} to prove the existence of the solution in the domain $\bar{D}$.
\end{proof}

\begin{theorem}
Suppose that $A(x,y,h)$, $B(x,y,h)$, $\varphi^{(-)}(x,y,h)$, and $\varphi^{(+)}(x,$ $ y,h)$ are functions of $x$, $y$, and $h$ with continuous first-order partial derivatives with respect to their arguments in the domain $\{ x\in\mathbb{R}, y\in\mathbb{R}, t\in [0,T] \}$ containing the initial curve
\[
    \Gamma : x=s_1, \quad y=s_2, \quad t=0, \quad h=h_{init}(s_1,s_2),
\]
If the function $h_{\text{init}}(s_1,s_2)$ is a continuously differentiable function of $s_1$ and $s_2$ on the closed interval $s_1, s_2 \in [0,1]$, then the problem \eqref{sec1.TransFuncEq} has a unique solution $h=h_0(x,y,t)$ in the domain $\{ x\in\mathbb{R}, y\in\mathbb{R}, t\in [0,T] \}$ for the initial curve $\Gamma$.
\end{theorem}
\begin{proof}
In the case of problem \eqref{sec1.TransFuncEq}, we let $a_1(x_1, x_2, x_3, u) = 1$, $a_2(x_1, x_2, x_3,$ $  u) = A(x,y,h)$, $a_3(x_1, x_2, x_3, u) = B(x,y,h)$, and $c(x_1, x_2, x_3, u) = -\frac{1}{2}(\varphi^{(-)}(x,$ $ y,h) + \varphi^{(+)}(x,y,h))$. The space $(t,x,y)$ is used instead of the $(x,y,z)$-space. The initial curve is denoted by $\Gamma$. 
The characteristic condition for this case can be written as
    \[
        J=\left.\operatorname{det}\left(\begin{array}{ccc}
	  1 & A & B \\
	0 & 1 & 0 \\
	0 & 0 & 1
    \end{array}\right)\right|_\Gamma = 1 \neq 0,
    \]
    Thus, by applying Theorem \ref{characteristictheorem}, we can prove the existence of the solution in the domain $\{ x\in\mathbb{R}, y\in\mathbb{R}, t\in [0,T] \}$.
\end{proof}

\end{document}